\documentclass[12pt]{amsart}
\usepackage{amssymb}
\usepackage{verbatim}
\usepackage{hyperref}
\usepackage{color}

\begin{document}

\newtheorem{theorem}{Theorem}    
\newtheorem{proposition}[theorem]{Proposition}
\newtheorem{conjecture}[theorem]{Conjecture}
\def\theconjecture{\unskip}
\newtheorem{corollary}[theorem]{Corollary}
\newtheorem{lemma}[theorem]{Lemma}
\newtheorem{sublemma}[theorem]{Sublemma}
\newtheorem{observation}[theorem]{Observation}
\theoremstyle{definition}
\newtheorem{definition}{Definition}
\newtheorem{notation}[definition]{Notation}
\newtheorem{remark}[definition]{Remark}
\newtheorem{question}[definition]{Question}
\newtheorem{questions}[definition]{Questions}
\newtheorem{example}[definition]{Example}
\newtheorem{problem}[definition]{Problem}
\newtheorem{exercise}[definition]{Exercise}

\numberwithin{theorem}{section}
\numberwithin{definition}{section}
\numberwithin{equation}{section}

\def\earrow{{\mathbf e}}
\def\rarrow{{\mathbf r}}
\def\uarrow{{\mathbf u}}
\def\varrow{{\mathbf V}}
\def\tpar{T_{\rm par}}
\def\apar{A_{\rm par}}

\def\reals{{\mathbb R}}
\def\torus{{\mathbb T}}
\def\heis{{\mathbb H}}
\def\integers{{\mathbb Z}}
\def\naturals{{\mathbb N}}
\def\complex{{\mathbb C}\/}
\def\distance{\operatorname{distance}\,}
\def\support{\operatorname{support}\,}
\def\dist{\operatorname{dist}\,}
\def\Span{\operatorname{span}\,}
\def\degree{\operatorname{degree}\,}
\def\kernel{\operatorname{kernel}\,}
\def\dim{\operatorname{dim}\,}
\def\codim{\operatorname{codim}}
\def\trace{\operatorname{trace\,}}
\def\Span{\operatorname{span}\,}
\def\dimension{\operatorname{dimension}\,}
\def\codimension{\operatorname{codimension}\,}
\def\nullspace{\scriptk}
\def\kernel{\operatorname{Ker}}
\def\ZZ{ {\mathbb Z} }
\def\p{\partial}
\def\rp{{ ^{-1} }}
\def\Re{\operatorname{Re\,} }
\def\Im{\operatorname{Im\,} }
\def\ov{\overline}
\def\eps{\varepsilon}
\def\lt{L^2}
\def\diver{\operatorname{div}}
\def\curl{\operatorname{curl}}
\def\etta{\eta}
\newcommand{\norm}[1]{ \|  #1 \|}
\def\expect{\mathbb E}
\def\bull{$\bullet$\ }

\def\xone{x_1}
\def\xtwo{x_2}
\def\xq{x_2+x_1^2}
\newcommand{\abr}[1]{ \langle  #1 \rangle}

\newcommand{\Norm}[1]{ \left\|  #1 \right\| }
\newcommand{\set}[1]{ \left\{ #1 \right\} }
\def\one{\mathbf 1}
\def\whole{\mathbf V}
\newcommand{\modulo}[2]{[#1]_{#2}}

\def\scriptf{{\mathcal F}}
\def\scriptg{{\mathcal G}}
\def\scriptm{{\mathcal M}}
\def\scriptb{{\mathcal B}}
\def\scriptc{{\mathcal C}}
\def\scriptt{{\mathcal T}}
\def\scripti{{\mathcal I}}
\def\scripte{{\mathcal E}}
\def\scriptv{{\mathcal V}}
\def\scriptw{{\mathcal W}}
\def\scriptu{{\mathcal U}}
\def\scriptS{{\mathcal S}}
\def\scripta{{\mathcal A}}
\def\scriptr{{\mathcal R}}
\def\scripto{{\mathcal O}}
\def\scripth{{\mathcal H}}
\def\scriptd{{\mathcal D}}
\def\scriptl{{\mathcal L}}
\def\scriptn{{\mathcal N}}
\def\scriptp{{\mathcal P}}
\def\scriptk{{\mathcal K}}
\def\frakv{{\mathfrak V}}
\def\Rn{{\mathbb{R}^n}}
\author{Juan Zhang}
\address{Juan Zhang \\
         School of Mathematical Sciences \\
         Beijing Normal University \\
         Laboratory of Mathematics and Complex Systems \\
         Ministry of Education \\
         Beijing 100875 \\
         People's Republic of China}
\email{zhangjuanw@mail.bnu.edu.cn}
\author{Hiroki Saito}
\address{College of Science and Technology\\ Nihon University\\
Narashinodai 7-24-1\\ Funabashi City, Chiba, 274-8501\\ Japan}
\email{saitou.hiroki@nihon-u.ac.jp}

\author{Qingying Xue}
\address{Qingying Xue
\\
School of Mathematical Sciences
\\
Beijing Normal University
\\
Laboratory of Mathematics and Complex Systems
\\
Ministry of Education
\\
Beijing 100875
\\
People's Republic of China
}
\email{qyxue@bnu.edu.cn}


\thanks{The third author was supported partly by NSFC
(No. 11471041, 11671039), the Fundamental Research Funds for the Central Universities (No. 2014KJJCA10) and NCET-13-0065\\
Corresponding author: Qingying Xue, Email: qyxue@bnu.edu.cn}



\date{June 30, 2017.}

\title
[The Fefferman-Stein type inequality]
{The Fefferman-Stein type inequalities for the multilinear strong maximal functions\\}

\maketitle

\begin{abstract}Let $\vec{\omega}=( \omega_{1},...,\omega_{m})$ be a multiple weight and  $\{\Psi_{j}\}^{m}_{j=1}$ be a sequence of Young functions. Let $\mathcal{M}_{\mathcal{R}}^{\vec{\Psi}}$ be the multilinear strong maximal function with Orlicz norms which is defined by
$$\mathcal{M}_{\mathcal{R}}^{\vec{\Psi}}(\vec{f})(x)=\sup_{R\in \mathcal{R},R\ni x}\prod^{m}_{j=1}\|f_{j}\|_{\Psi_{j},R}$$
where the supremum is taken over all rectangles with sides parallel to the coordinate axes. If $\Psi_j(t)=t$, then $\mathcal{M}_{\mathcal{R}}^{\vec{t}}$ coincides with the multilinear strong maximal function $\mathcal{M}_{\mathcal{R}}$ defined and studied by Grafakos et al. In this paper, we first investigated the Fefferman-Stein type inequality for $\mathcal{M}_{\mathcal{R}}^{\vec{\Psi}}$ when $\vec{\omega}$ satisfies the $A_{\infty,\mathcal{R}}$ condition. Then, for arbitrary $\vec{\omega}\geq 0$( each $ \omega_{j}\ge 0$), the Fefferman-Stein type inequality for  the multilinear strong maximal function $\mathcal{M}_{\mathcal{R}} $ associated with rectangles will be given.
\end{abstract}
\section{Introduction}

\subsection{Hardy-Littlewood and strong maximal functions}

Let $f$ be a locally integrable function defined on $\mathbb{R}^{n}$ and $\mathcal{Q}$ be the family of all cubes in $\mathbb{R}^{n}$ with sides parallel to the coordinate axes. Let ${M}$ be the classical Hardy-Littlewood maximal function defined by
\begin{equation}{M}f(x)=\sup_{Q\in\mathcal{Q},Q\ni x}\frac{1}{|Q|}\int_{Q}|f(y)|dy.\label{HL}\end{equation}
 It  was well known that ${M}$ is of weak type $(1,1)$ and strong type $(p,p)$ for $p>1$. Moreover, for arbitrary weight $\omega$, it was shown by Fefferman and Stein \cite{FS} that ${M}$ enjoys the following property:\begin{equation}\omega(\{x\in \mathbb{R}^{n}:{M}f(x)>t\})\leq \frac{C}{t}\|f\|_{L^{1}(\mathbb{R}^{n},{M}\omega)},\quad  t>0.\label{FS1}\end{equation}
 By interpolation, it gives immediately that
\begin{equation}\|{M}f\|_{L^{p}(\mathbb{R}^{n},\omega)}\leq C\|f\|_{L^{p}(\mathbb{R}^{n},{M}\omega)}, \quad p>1.\label{FS2}\end{equation}
Inequalities in (\ref{FS1}) and (\ref{FS2}) are all called the Fefferman-Stein type inequalities.

Instead of cubes, more general geometry structure has been assigned to the operator ${M}$. For example, if the family $\mathcal{Q}$ is replaced by $\mathcal{R}$, the family of all rectangles in $\mathbb{R}^{n}$ with sides parallel to the coordinate axes, then the maximal function becomes the well known strong maximal function as follows :
$${M}_{\mathcal{R}}f(x)=\sup_{R\in\mathcal{R},R\ni x}\frac{1}{|R|}\int_{R}|f(y)|dy.$$ In 1935, a maximal
theorem for ${M}_{\mathcal{R}}$ was given by Jessen, Marcinkiewicz and Zygmund  \cite{JMZ}. They showed that ${M}_{\mathcal{R}}$ is not of weak
type $(1,\,1)$, which is quite different from the properties of the classical Hardy-Littlewood maximal function. As a replacement of the weak
$(1,\,1)$ estimate, it was demonstrated in \cite{JMZ} that ${M}_{\mathcal{R}}$ enjoys the following end-point behavior property:
\begin{equation}\label{endpoint-JMZ}
\big|\{x \in \Rn; M_{\mathcal{R}}f(x)>\lambda \}\big|
\lesssim_{n} \int_{\Rn} \frac{|f(x)|}{\lambda} \left(1 + \Big(\log^+ \frac{|f(x)|}{\lambda}\Big)^{n-1}\right) dx.
\end{equation}
In 1975, C\'{o}rdoba and Fefferman \cite{CF} gave a geometric proof of $(\ref{endpoint-JMZ})$ and established a covering lemma for rectangles. Their covering lemma is quite useful by the reason that it overcomes the failure of the Besicovitch covering argument for rectangles with arbitrary eccentricities. Subsequently, achievements have been made to obtain the corresponding weighted version of $(\ref{endpoint-JMZ})$. Among those achievements are the nice works of Bagby and Kurtz \cite{BK}, Capri and Guti\'{e}rrez \cite{CG}, Mitsis \cite{M}, Luque and Parissis \cite{LP}. In \cite{LP}, Luque and Parissis formulated a weighted version of C\'{o}rdoba-Fefferman's covering lemma and showed that the following weighted inequality holds for $\omega\in A_{\infty, \mathcal{R}}$:
\begin{equation}\label{Weighted}
\omega\big(\{x \in \Rn; M_{\mathcal{R}}f(x)>\lambda \}\big)
\lesssim_{\omega,n} \int_{\Rn} \frac{|f(x)|}{\lambda} \left(1 + \Big(\log^+ \frac{|f(x)|}{\lambda}\Big)^{n-1}\right) M_{\mathcal{R}} \omega(x) dx.
\end{equation}
Recently, for $n=2$, the condition $\omega\in A_{\infty, \mathcal{R}}$ in (\ref{Weighted}) was extended to any weight $\omega\ge 0$ by Saito and Tanaka \cite{ST} as follows:
$$\omega(\{x\in \mathbb{R}^{2}:{M}_{\mathcal{R}}f(x)>t\})\leq C\int_{\mathbb{R}^{2}}\frac{|f(x)|}{t}\left(1+\log^{+}\frac{|f(x)|}{t}\right)W(x)dx, t>0,$$
where $W={M}_{\mathcal{R}}{M}_{\mathcal{Q}}\omega$ and the constant $C>0$ does not depend on $\omega$ and $f$.

Still more recently, Tanaka \cite{T} further essentially extended the results in \cite{ST} to higher dimensions.
We summarize the results in  \cite{T} as follows:

{\noindent \textbf {Theorem A. }(\cite{T}). {\it {For $p>1$ and any weight $\omega$ defined on $\mathbb{R}^{n}$, there exists a constant $C>0$ which does not depend on $\omega$ and $f$, such that the following inequality holds
\begin{equation}\label{ta}
\omega(\{x\in \mathbb{R}^{n}:{M}_{\mathcal{R}}f(x)>t\})^{1/p}\leq \frac{C}{t}\|f\|_{L^{p}(\mathbb{R}^{n},W)}, \quad \hbox{for\ }t>0
\end{equation}where $W={M}_{\mathcal{R}}{M}_{\mathcal{R}}^{n-1}...{M}_{\mathcal{R}}^1\omega$ and $M_{\mathcal{R}}^c$ $(c=1,...,n-1)$ is the strong maximal operator with the complexity $c$ defined in Section \ref{NP}. }}

\subsection{Multilinear strong maximal functions}

In order to state more clearly, we first introduce one definition.
\begin{definition} [\bf Multilinear strong maximal function with Orlicz norms, \cite{LL}] Let $\vec{f}=(f_{1},..., f_{m})$ be an $m$-dimensional vector of locally integrable functions. The multilinear strong maximal function with Orlicz norms is defined by
$$\mathcal{M}_{\mathcal{R}}^{\vec{\Psi}}(\vec{f})(x)=\sup_{R\in \mathcal{R},R\ni x}\prod^{m}_{j=1}\|f_{j}\|_{\Psi_{j},R}$$
where $\{\Psi_{j}\}^{m}_{j=1}$ is a sequence of Young functions and the supremum is taken over all rectangles with sides parallel to the coordinate axes.
\end{definition}
\begin{remark}
In particular, if $\Psi_{j}(t)=t$, for all $t\in (0,\infty)$ and all $j\in \{1,...,m\}$, $\mathcal{M}_{\mathcal{R}}^{\vec{\Psi}}$ coincides with the multilinear strong maximal function $\mathcal{M}_{\mathcal{R}}$ introduced and studied by  Grafakos et al. \cite{GLPT} in 2011. The authors \cite{GLPT} demonstrated that $\mathcal{M}_{\mathcal{R}}$ still enjoys a similar endpoint $L\log L $ type estimate as follows:
for any $\lambda > 0$
\begin{equation}\label{Endpoint-Grafakos}
\Big| \big\{x \in \Rn; \mathcal{M}_{\mathcal{R}}(\vec{f})(x) > \lambda^m \big\} \Big|
\lesssim_{m,n} \bigg(\prod_{i=1}^m  \int_{\Rn} \Phi_n^{(m)} \left(\frac{|f_i(y)|}{\lambda}\right) dy\bigg)^{1/m},
\end{equation}
where $\Phi_n(t):=t[1+(\log ^+ t)^{n-1}]$ $(t>0)$ and $\Phi_n^{(m)}$ is
$m$-times compositions of the function $\Phi_n$ with itself.
Furthermore, the exponent is sharp in the sense that we cannot replace
$\Phi_n^{(m)}$ by $\Phi_n^{(k)}$ for $k \leq m-1$.

\end{remark}
This paper will be devoted to investigate the Fefferman-Stein type inequalities for the multilinear strong maximal functions. The first main results of this paper concerns with the multilinear strong maximal functions with Orlicz norms $\mathcal{M}_{\mathcal{R}}^{\vec{\Psi}}$.
\begin{theorem} \label{tm1}
Let $1<p_{1},...,p_{m}<\infty$ such that $\frac{1}{p}=\sum^{m}_{j=1}\frac{1}{p_{j}}.$ Assume that $\mathcal{R}$ is a basis and $\{\Psi_{j}\}^{m}_{j=1}$ is a sequence of Young functions such that $\Psi_{j}\in B^{*}_{p_{j}}$. Let $\vec{\omega}=(\omega_{1},...,\omega_{m})$ and $\nu_{\vec{\omega}}=\prod^{m}_{j=1}\omega_{j}^{p/p_{j}}\in A_{\infty,\mathcal{R}}$, then there exists a constant $C>0$ such that for all nonnegative functions $f$, the following inequality holds
\begin{equation}\label{INEQ}\int_{\mathbb{R}^{n}}[\mathcal{M}_{\mathcal{R}}^{\vec{\Psi}}(\vec{f})(y)]^{p}\nu_{\vec{\omega}}(y)dy\leq C\prod^{m}_{j=1}\|f_{j}\|^{p}_{L^{p_{j}}(\mathcal{M}_{\mathcal{R}}\omega_{j})}.\end{equation}
\end{theorem}
Note that in Theorem \ref{tm1}, we need to assume that $\nu_{\vec{\omega}}\in A_{\infty,R}$. For arbitrary
weights, the methods to establish the Fefferman-Stein type inequalities are quite different from Theorem \ref{tm1}. Moreover, $\mathcal{M}_{\mathcal{R}}\omega_{j}$ in (\ref{INEQ}) will be replaced by more larger maximal functions. For simplicity, we only consider the multilinear strong maximal operator $\mathcal{M}_{\mathcal{R}}.$
\begin{theorem}\label {tm2}
Let $1<p_{1},...,p_{m}<\infty$ and $\sum^{m}_{i=1}\frac{1}{p_{i}}=\frac{1}{p}$.  Let $\vec{\omega}=( \omega_{1},...,\omega_{m})$ and suppose that each $\omega_{j}$ is an arbitrary
weight. Denote by
$W_{j}$=$\mathcal{M} _{\mathcal{R}}\mathcal{M}_{\mathcal {R}}^{n-1}...\mathcal{M}_{\mathcal {R}}^{1}\omega_{j}$ and Set $\nu_{\vec{\omega}}=\prod^{m}_{j=1}\omega_{j}^{{p}/{p_{j}}}$. Then, there exists a positive constant $C$ which does not depend on $\omega_{j}$ and $f_{j}$, such that the following inequality holds
$$\nu_{\vec{\omega}}\left({x\in\mathbb{R}^{n}:\mathcal{M}_{\mathcal{R}} (\vec{f})(x)> t^{m}}\right)^{1/p}\leq \prod^{m}_{j=1}\frac{C}{t}\|f_{j}\|_{L^{p_{j}}(\mathbb{R}^{n},W_{j})}.$$
\end{theorem}
By interpolation, Theorem \ref{tm2} yields the following corollary.\begin{corollary}
Let $1<p_{1},...,p_{m}<\infty$ and $\sum^{m}_{i=1}\frac{1}{p_{i}}=\frac{1}{p}$. Given $\vec{\omega}=( \omega_{1}, \omega_{2},...,\omega_{m})$, where each $\omega_{j}$ is an arbitrary
weight. Set $\nu_{\vec{\omega}}=\prod^{m}_{j=1}\omega_{j}^{{p}/{p_{j}}}$ and
$W_{j}$=$\mathcal{M} _{\mathcal{R}}\mathcal{M}_\mathcal{R}^{n-1}...\mathcal{M}_\mathcal{R}^{1}\omega_{j}$. Then, there exists a positive constant $C$ which does not depend on $\omega_{j}$ and $f_{j}$, such that the following inequality holds \begin{equation}\int_{\mathbb{R}^{n}}[\mathcal{M}_{\mathcal{R}}(\vec{f})(y)]^{p}\nu_{\vec{\omega}}(y)dy\leq C\prod^{m}_{j=1}\|f_{j}\|^{p}_{L^{p_{j}}(\mathbb{R}^{n},W_{j})}\end{equation}\end{corollary}
 \section{notions and preliminaries}\label{NP}First, we give the definitions of two kinds of maximal functions.
\begin{definition}[\bf {multilinear maximal operator with cubes},\cite{LOPTT}]
Given $\vec{f}=(f_{1},f_{2},...,f_{m})$, we define the maximal operator $\mathcal{M}$ by
$$\mathcal{M}(\vec{f})(x)=\sup_{Q\ni x}\prod^{m}_{j=1}\frac{1}{|Q|}\int_{Q}|f_{j}(y)|dy,$$
where the supremum is taken over all cubes $Q$ containing $x$, with sides parallel to the coordinate axes .
\end{definition}
\begin{definition}[\bf {Strong maximal operator with complexity $c$}, \cite{T}]
Let $c=1,2,...,n$. We say that the set of rectangles $\mathcal{R}$ in $\mathbb{R}^{n}$ have the complexity $c$ whenever the side lengths of $R$ are exactly $\alpha_{1}$ or $\alpha_{2}$... or $\alpha_{c}$ for varying $\alpha_{1}$, $\alpha_{2}$,... or $\alpha_{c}>0$. That is, the set of rectangles with complexity $c$ is the $c$-parameter family of rectangles. For a locally integrable function $f$ on $\mathbb{R}^{n}$, the strong maximal operator with complexity $c$ is defined by
$${M}_{\mathcal {R}}^cf(x)=\sup_{R\in \mathcal{R}_{c},R\ni x}\frac{1}{|R|}\int_{R}|f(y)|dy,$$
where $\mathcal{R}_{c}$ is the set of all rectangles in $\mathbb{R}^{n}$, with sides parallel to the coordinate axes and having the complexity $c$.\end{definition}
Then we can define the multilinear setting of it. That is,
\begin{definition}[\bf {Multilinear strong maximal operator with complexity $c$}]
Let $c=1,2,...,n$, and $\vec{f}=(f_{1},...,f_{m})$ is an $m$-dimensional vector of locally integrable functions, the strong maximal operator $\mathcal{{M}}_{\mathcal {R}}^{c}(\vec{f})$ is defined by \begin{equation}\mathcal{{M}}_{\mathcal {R}}^{c}{f})(x)=\sup_{R\in \mathcal{R}_{c}, R\ni x}\prod^{m}_{j=1}\frac{1}{|R|}\int_{R}|f_{j}(y)|dy,\end{equation}
where $\mathcal{R}_{c}$ is the set of all rectangles in $\mathbb{R}^{n}$, with sides parallel to the coordinate axes and having the complexity $c$.
\end{definition}
\remark If $c=n$, it is easy to check that ${M}_{\mathcal {R}}^c$ coincides with the strong maximal function $M_{\mathcal {R}}$, and  $\mathcal{{M}}_{\mathcal {R}}^{c}$ coincides with the multilinear strong maximal operators $\mathcal{{M}}_{\mathcal {R}}$.

\subsection {Basic facts about weights}For $1<p<\infty$, a weight $\omega$ associated with $\mathcal{R}$ is said to satisfy the $A_{p, \mathcal{R}}$ condition, if it holds that
\begin{equation*}
\sup_{R\in\mathcal{R}}\left(\frac{1}{|{R}|}\int_{{R}}\omega dx\right)\left(\frac{1}{|R|}\int_{R}\omega^{1-p'}dx\right)^{\frac{p}{p'}}<\infty.
\end{equation*}
In the case $p=1$, we say that $\omega$ satisfies the $A_{1,\mathcal{R}}$ condition if ${M}_{\mathcal{R}}\omega(x)\leq c\omega(x)$ for almost all $x\in \mathbb{R}^{n}$. It follows from these definitions and the H\"{o}lder inequality that $A_{p,\mathcal{R}}\subset A_{q,\mathcal{R}}$ if $1\leq p\leq q< \infty$. Then it is natural to define the class $A_{\infty, \mathcal{R}}$ by setting $A_{\infty,\mathcal{R}}=\bigcup_{p>1}A_{p,\mathcal{R}}$. Recall that $\omega$ is said to satisfy Condition (A) \cite{HLP} if there are constants $0<\lambda<1$, $0<c(\lambda)<\infty$ such that for all measurable sets $E$, it holds that $\omega(\{x\in \mathbb{R}^{n}: \mathcal{M}_{\mathcal{R}}[\chi_{E}](x)>\lambda\})\leq c(\lambda)\omega(E).$ A basic fact is presented by Hagelstein, and Parissis \cite{HP} that the asymptotic estimate for the constant in Condition (A) is equivalent to $\omega \in A_{\infty, \mathcal{R}}$.

The multiple version of $A_{p, \mathcal{R}}$ is defined as follows:
\begin{definition}[\cite{GLPT}]
Let $1\leq p_{1},...,p_{m}<\infty$. Given $\vec{\omega}=(\omega_{1},...,\omega_{m})$, set
$\nu_{\vec{\omega}}=\prod^{m}_{i=1}\omega_{i}^{p/p_{i}}.$
The $m$-tuple weight $\vec{\omega}$ associated with $\mathcal{R}$ is said to satisfy the $A_{\vec{p}, \mathcal{R}}$ condition if
\begin{equation*}
\sup_{R\in \mathcal{R}}\left(\frac{1}{|R|}\int_{R}\nu_{\vec{\omega}} dx\right)\prod^{m}_{j=1}\left(\frac{1}{|R|}\int_{R}\omega_{j}^{1-p'_{j}}dx\right)^{\frac{p}{p'_{j}}}<\infty.
\end{equation*}
When $p_{j}=1$, $(\frac{1}{|R|}\int_{R}\omega_{j}^{1-p'_{j}})^{1/p'_{j}}$ is understood as $(\inf_{R}\omega_{j})^{-1}$.
\end{definition}
\subsection {Basic facts about Young functions}
First, we need to recall some definitions and basic facts about Young functions.
\begin{definition}[\cite{GLPT}]
A Young function is a continuous, convex, increasing function $\Phi : [0,\infty]\rightarrow [0,\infty]$ with $\Phi(0)=0$ and $\Phi(t)\rightarrow\infty$ as $t\rightarrow\infty$. For $0<\varepsilon<1$ and $t\geq 0$,
the properties of $\Phi$ easily imply that \begin{equation*}
\Phi(\varepsilon t)\leq \varepsilon \Phi(t).
\end{equation*}
The $\Phi$-norm of a function $f$ over a set $E$ with finite measure is defined by
\begin{equation*}
\|f\|_{\Phi, E}=\inf\left\{\lambda>0:\frac{1}{|E|}\int_{E}\Phi(\frac{|f(x)|}{\lambda})dx\leq1\right\}.
\end{equation*}
Associated with each Young function $\Phi$, one can define its complementary function
\begin{equation*}
\bar{\Phi}(s)=\sup_{t>0}\{st-\Phi(t)\}, \quad \hbox{\ for s\ }\geq0.
\end{equation*}
It is well known that $\bar{\Phi}$-norms are related to the $L_{\Phi}$-norms via the following generalized H\"{o}lder inequality:
\begin{equation*}
\frac{1}{|E|}\int_{E}|f(x)g(x)|dx\leq2||f||_{\Phi, E}||g||_{\bar{\Phi}, E}.
\end{equation*}
\end{definition}
\begin{definition}[\bf{Strong $B^{\ast}_{p}$ condition}, \cite{LL}]
Let $1<p<\infty$. A Young function $\Phi$ is said to satisfy the strong $B^{\ast}_{p}$ condition, or  $\Phi\in B^{\ast}_{p}$, if there is a positive constant $c$ such that the following inequality holds
\begin{equation*}
\int^{\infty}_{c}\frac{\Phi_{n}(\Phi(t))}{t^{p}}\frac{dt}{t}<\infty,
\end{equation*}
where $\Phi_{n}(t): = t[\log(e+t)]^{n-1}\sim t[1+(\log^{+}t)^{n-1}]$ for all $t> 0$. \end{definition}

\section {The F-S inequality with weights in $A_{\infty,\mathcal{R}}$}}
In this section, we give the proof of Theorem 1.1, first we give two lemmas which play an important role in our proof.
\begin{lemma}
Let $1<p_{1},...,p_{m}<\infty$ and $0<p<\infty$ such that $\frac{1}{p}=\sum^{m}_{j=1}\frac{1}{p_{j}}.$ Assume that $\mathcal{R}$ is a basis and that $\{\Psi_{j}\}^{m}_{j=1}$ is a sequence of Young functions such that $\Psi_{j}\in B^{*}_{p_{j}}$, then, $\mathcal{M}_{\mathcal{R}}^{\vec{\Psi}}$ is bounded from $L^{p_{1}}(\mathbb{R}^{n})\times L^{p_{2}}(\mathbb{R}^{n})\times...\times L^{p_{m}}(\mathbb{R}^{n})$ to $ L^{p}(\mathbb{R}^{n})$.
\end{lemma}
\begin{proof}
Let $\mathbf{M}_{\mathcal{R}}^{\Psi}$ be the Orlicz maximal operator on $\mathbb{R}^{n}$ defined by
\begin{equation*}
\mathbf{M}_{\mathcal{R}}^{\Psi}(f)(x)=\sup_{R\in \mathcal{R}, R\ni x}\|f\|_{\Psi, R},
\end{equation*}
where the supremum is taken over all rectangles with sides parallel to the coordinate axes.

Observing that for all $x\in \mathbb{R}^{n}$ and for all nonnegative functions $\vec{f}=(f_{1},...,f_{m})$, multilinear Orlicz maximal function is controlled by the $m$-fold tensor product of the Orlicz maximal function of each variable. That is,
\begin{equation*}
\mathcal{M}_{\mathcal{R}}^{\vec{\Psi}}(\vec{f})(x)\leq\prod^{m}_{j=1}\mathbf{M}_{\mathcal{R}}^{\Psi_{j}}(f_{j})(x).
\end{equation*}
Since $\Psi_{j}\in B^{*}_{p_{j}}$, it follows that every $\mathbf{M}_{\mathcal{R}}^{\Psi_{j}}$ is bounded on $L^{p_{j}} (\mathbb{R}^{n})$ (\cite{LL},Theorem 2.1). This yields immediately that $\mathcal{M}_{\mathcal{R}}^{\vec{\Psi}}$ is bounded from $L^{p_{1}}(\mathbb{R}^{n})\times L^{p_{2}}(\mathbb{R}^{n})\times...\times L^{p_{m}}(\mathbb{R}^{n})$ to $L^{p}(\mathbb{R}^{n})$.
\end{proof}
\begin{definition}[\cite{GLPT}]
Let $\mathcal{R}$ be a basis and let $0<\alpha<1$. A finite sequence $\{\tilde{A}_{i}\}^{M}_{i=1}\subset \mathcal{R}$ of sets of finite $dx$-measure is called $\alpha$-scattered with respect to the Lebesgue measure if
$$\left|\tilde{A}_{i}\cap \bigcup_{s<i}\tilde{A}_{s}\right|\leq \alpha|\tilde{A}_{i}|,\quad\quad \hbox{for \ all }\ 1<i\leq M.$$
\end{definition}
\begin{lemma}[\cite{GLPT}]
Let $\mathcal{R}$ be a basis and let $\omega$ be a weight associated with this basis. Suppose further that $\omega$ satisfies condition (A) for some $0<\lambda<1$ and $0<c(\lambda)<\infty$. Then given any finite sequence $\{A_{i}\}^{M}_{i=1}$ of sets $A_{i}\in \mathcal{R}$, it holds that
\begin{enumerate}
\item [(1)] we can find a subsequence $\{\tilde{A}_{i}\}_{i\in I}$ of $\{A_{i}\}^{M}_{i=1}$ which is $\lambda$-scattered with respect to the Lebesgue measure;\\
\item [(2)] $\tilde{A}_{i}=A_{i}$, $i\in I$;\\
\item [(3)] for any $1\leq i<j\leq M+1$
$$\omega\left(\bigcup_{s<j}A_{s}\right)\leq c(\lambda)\left[\omega\left(\bigcup_{s<i}A_{s}\right)+\omega\left(\bigcup_{i\leq s<j}\tilde{A}_{s}\right)\right]$$
\end{enumerate}\end{lemma}
Now, we are in the position to give the proof of Theorem \ref{tm1}.
\begin{proof}
The argument we will employ here is essentially a combination of the ideas from \cite{GLPT}, \cite{J}, \cite{LL}. Let $N>0$ be a large integer. We will prove the required estimate for the quantity
$$\int_{2^{-N}<\mathcal{M}^{\vec{\Psi}}_{\mathcal{R}}(\vec{f})\leq 2^{N+1}}\mathcal{M}^{\vec{\Psi}}_{\mathcal{R}}(\vec{f})(x)^{p}\nu_{\vec{\omega}}(x)dx$$
with a bound  independent of $N$. First, for each integer $k$, $|k|\leq N$, there exist a compact set
$$K_{k}\subset\left\{\mathcal{M}^{\vec{\Psi}}_{\mathcal{R}}(\vec{f})(x)>2^{k}\right\}$$
satisfying
$$\nu_{\vec{\omega}}(K_{k})\leq \nu_{\vec{\omega}}(\{\mathcal{M}^{\vec{\Psi}}_{\mathcal{R}}(\vec{f})(x)>2^{k})\leq 2 \nu_{\vec{\omega}}(K_{k})$$
and a finite sequence $b_{k}=\{B_{r}^{k}\}_{r\geq 1}$ of sets $B_{r}^{k}\in \mathcal{R}$ with
$$\prod_{j=1}^{m}\|f_{j}\|_{\Psi_{j},B^{k}_{r}}> 2^{k}.$$
We set $b_{k}=\emptyset$ if $|k|>N$ and
$$\Omega_{k}=\left\{
  \begin{aligned}
    \bigcup_{r}B^{k}_{r}&, & |k|\leq N,\\
    \emptyset                                     &, & |k|>N.
  \end{aligned}
\right.
$$
Observe that these sets are decreasing in $k$, i.e.,$\Omega_{k+1}\subset \Omega_{k}$.
We now distribute the sets in $\bigcup_{k}b_{k}$ over $\mu$ sequences $\{A_{i}(l)\}_{i\geq1}, 0\leq l\leq\mu-1$,
where $\mu$ will be chosen momentarily to be an appropriately large natural number. Set $i_{0}(0)=1$. In the first $i_{1}(0)-i_{0}(0)$ entries of $\{A_{i}(0)\}_{i\geq1,}$, i.e., for
$$i_{0}(0)\leq i<i_{1}(0),$$
we place the elements of the sequence $b_{N}=\{B^{N}_{r}\}_{r\geq1}$ in the order indicated by the index $r$. For the next $i_{2}(0)-i_{1}(0)$ entries of $\{A_{i}(0)\}_{i\geq1,}$, i.e., for
$$i_{1}(0)\leq i<i_{2}(0),$$
we place the elements of the sequence $b_{N-\mu}$. Continue in this way until we reach the first integer $m_{0}$ such that $N-m_{0}\mu\geq-N$, when we stop. For indices $i$ satisfying
$$i_{m_{0}}(0)\leq i<i_{m_{0}+1}(0),$$
we place in the sequence $\{A_{i}(0)\}_{i\geq1}$ the elements of $b_{N-m_{0}\mu}$. The sequences $\{A_{i}(l)\}_{i\geq1}$, $1\leq l\leq\mu-1$, are defined similarly, starting from $b_{N-l}$ and using the families $b_{N-l-s\mu}, s=0, 1, ..., m_{l}$, where $m_{l}$ is chosen so that $N-l-m_{l}\mu\geq-N$.\par
Since $\nu_{\vec{\omega}}\in A_{\infty,\mathcal{R}}, \nu_{\vec{\omega}}$ satisfies condition $(A)$, and we may apply Lemma 3.2 to each $\{A_{i}(l)\}_{i\geq1}$ for some fixed $0<\lambda<1$. Then we obtain sequences
$$\{\tilde{A}_{i}(l)\}_{i\geq1}\subset\{A_{i}(l)\}_{i\geq1}, 0\leq l\leq\mu-1,$$
which are $\lambda$-scattered with respect to the Lebesgue measure. In view of the definition of the set $\Omega_{k}$ and the construction of the families $\{A_{i}(l)\}_{i\geq1}$, we can use assertion (3) of Lemma 3.2 to obtain
$$
\aligned
\nu_{\vec{\omega}}(\Omega_{k})
&\leq c\left[\nu_{\vec{\omega}}(\Omega_{k+\mu})+\nu_{\vec{\omega}}\left(\bigcup_{i_{m_{l}}\leq i<i_{m_{l+1}}}\tilde{A}_{i}(l)\right)\right]\\
&\leq c\nu_{\vec{\omega}}(\Omega_{k+\mu})+c\sum^{i_{m_{l+1}}(l)-1}_{i=i_{m_{l}}(l)}\nu_{\vec{\omega}}(\tilde{A}_{i}(l))
\endaligned
$$
if $k=N-l-m_{l}\mu$. It will be enough to consider these indices $k$ because the sets $\Omega_{k}$ are decreasing.\par
Now all the sets $\{{\tilde{A}_{i}}(l)\}^{i_{m+1}(l)-1}_{i=i_{m}(l)}$ belong to $b_{k}$ with $k=N-l-m_{l}\mu$, and therefore
$$\prod_{j=1}^{m}\|f_{j}\|_{\Psi_{j},{\tilde{A}_{i}}(l)}> 2^{k}.$$
Hence, it follows that
$$\int_{2^{-N}<\mathcal{M}^{\vec{\Psi}}_{\mathcal{R}}(\vec{f})\leq 2^{N+1}}\mathcal{M}^{\vec{\Psi}}_{\mathcal{R}}(\vec{f})(x)^{p}\nu_{\vec{\omega}}(x)dx\leq 2^{p}\sum_{k}2^{kp}\nu_{\vec{\omega}}(\Omega_{k}):=I_{1}$$
and then
$$
\aligned
&I_{1}\leq C\sum_{k}2^{kp}\nu_{\vec{\omega}}(\Omega_{k+\mu})+C\sum^{\mu-1}_{l=0}\sum_{i\in I(l)}\nu_{\vec{\omega}}({\tilde{A}_{i}}(l))\left(\prod_{j=1}^{m}\|f_{j}\|_{\Psi_{j},{\tilde{A}_{i}}(l)}\right)^{p}\\
&\quad=C2^{-p\mu}\sum_{k}2^{kp}\nu_{\vec{\omega}}(\Omega_{k})+C\sum^{\mu-1}_{l=0}\sum_{i\in I(l)}\nu_{\vec{\omega}}({\tilde{A}_{i}}(l))\left(\prod_{j=1}^{m}\|f_{j}\|_{\Psi_{j},{\tilde{A}_{i}}(l)}\right)^{p}\\
\endaligned
$$
If we choose $\mu$ so large that $C2^{-\mu p}\leq \frac{1}{2}$, and since everything involved is finite, the first term on the right-hand side can be subtracted from the left-hand side. This yields that
$$
\aligned
&\int_{2^{-N}<\mathcal{M}^{\vec{\Psi}}_{\mathcal{R}}(\vec{f})\leq 2^{N+1}}\mathcal{M}^{\vec{\Psi}}_{\mathcal{R}}(\vec{f})(x)^{p}\nu_{\vec{\omega}}(x)dx\\
&\leq 2^{p+1}C\sum^{\mu-1}_{l=0}\sum_{i\in I(l)}\nu_{\vec{\omega}}({\tilde{A}_{i}}(l))\left(\prod_{j=1}^{m}\|f_{j}\|_{\Psi_{j},{\tilde{A}_{i}}(l)}\right)^{p}\\
&\leq 2^{p+1}C\sum^{\mu-1}_{l=0}\sum_{i\in I(l)}\frac{\nu_{\vec{\omega}}({\tilde{A}_{i}}(l))}{|{\tilde{A}_{i}}(l)|}\left(\prod_{j=1}^{m}\|f_{j}\|_{\Psi_{j},{\tilde{A}_{i}}(l)}\right)^{p}|{\tilde{A}_{i}}(l)|.
\endaligned
$$
Since $\nu_{\vec{\omega}}=\prod_{j=1}^{m}\omega_{j}^{p/p_{j}}$, applying the H\"{o}lder inequality, we have
$$
\aligned
\frac{\nu_{\vec{\omega}}({\tilde{A}_{i}}(l))}{|{\tilde{A}_{i}}(l)|}
&=\frac{1}{|{\tilde{A}_{i}}(l)|}\int_{{\tilde{A}_{i}}(l)}\prod_{j=1}^{m}\omega_{j}^{p/p_{j}}dx\\
&\leq \frac{1}{|\tilde{A}_{i}(l)|}\prod_{j=1}^{m}\left(\int_{\tilde{A}_{i}(l)}\omega_{j}dx\right)^{p/p_{j}}\\
&= \prod_{j=1}^{m}\left(\frac{\omega_{j}(\tilde{A}_{i}(l))}{|\tilde{A}_{i}(l)|}\right)^{p/p_{j}}.
\endaligned
$$
Thus, we have
\begin{equation}\label{B}
\aligned
&2^{p+1}C\sum^{\mu-1}_{l=0}\sum_{i\in I(l)}\frac{\nu_{\vec{\omega}}(\tilde{A_{i}}(l))}{|\tilde{A}_{i}(l)|}\left(\prod_{j=1}^{m}\|f_{j}\|_{\Psi_{j},\tilde{A_{i}}(l)}\right)^{p}|\tilde{A_{i}}(l)|\\
&\leq 2^{p+1}C\sum^{\mu-1}_{l=0}\sum_{i\in I(l)}\prod_{j=1}^{m}\left\|f_{j}\left(\frac{\nu_{\omega_{j}}(\tilde{A_{i}}(l))}{|\tilde{A_{i}}(l)|}\right)^{1/p_{j}}\right\|^{p}_{\Psi_{j},\tilde{A_{i}}(l)}|\tilde{A_{i}}(l)|\\
&\leq 2^{p+1}C\sum^{\mu-1}_{l=0}\sum_{i\in I(l)}\prod_{j=1}^{m}\left\|f_{j}\left(\mathcal{M}_{\mathcal{R}}\omega_{j}\right)^{1/p_{j}}\right\|^{p}_{\Psi_{j},\tilde{A_{i}}(l)}|\tilde{A_{i}}(l)|.
\endaligned
\end{equation}
For each $l$, let
$E_{1}(l)=\tilde{A}_{1}(l)$ and $E_{i}(l)=\tilde{A}_{i}(l)\setminus\cup \tilde{A}_{s}(l)$, $i>1$. Recall that the sequences $a(l)=\{\tilde{A}_{i}(l)\}_{i\in I(l)}$ are $\lambda$-scattered with respect to the Lebesgue measure. Hence, it holds that
$$|\tilde{A}_{i}(l)|\leq\frac{1}{1-\lambda}|E_{i}(l)|, i>1.$$
Therefore,  (\ref{B}) can be further controlled by
\begin{equation}\label{A}
\frac{C}{1-\lambda}\sum^{\mu-1}_{l=0}\sum_{i\in I(l)}\prod_{j=1}^{m}\left\|f_{j}\left(\mathcal{M}_{\mathcal{R}}\omega_{j}\right)^{1/p_{j}}\right\|^{p}_{\Psi_{j},\tilde{A_{i}}(l)}|\tilde{E_{i}}(l)|.
\end{equation}
Now since the family $\{E_{i}(l)\}_{i,l}$ consists of pairwise disjoint sets, we can therefore apply Lemma 3.1  to estimate the inequality (\ref{A}). Hence,
$$
\aligned
&\frac{C}{1-\lambda}\sum^{\mu-1}_{l=0}\sum_{i\in I(l)}\prod_{j=1}^{m}\left\|f_{j}\left(\mathcal{M}_{\mathcal{R}}\omega_{j}\right)^{1/p_{j}}\right\|^{p}_{\Psi_{j},\tilde{A_{i}}(l)}|\tilde{E_{i}}(l)|\\
&\leq C\int_{\mathbb{R}^{n}}\mathcal{M}^{\vec{\Psi}}_{\mathcal{R}}\left(f_{1}(\mathcal{M}_{\mathcal{R}}\omega_{1})^{1/p_{1}},...,f_{m}(\mathcal{M}_{\mathcal{R}}\omega_{m})^{1/p_{m}}\right)(x)^{p}dx\\
&\leq C\prod^{m}_{j=1}\left\|f_{j}(\mathcal{M}_{R}\omega_{j})^{1/p_{j}}\right\|^{p}_{L^{p_{j}}(\mathbb{R}^{n})}=C\prod^{m}_{j=1}\left\|f_{j}\right\|^{p}_{L^{p_{j}}(\mathcal{M}_{\mathcal{R}}\omega_{j})}.
\endaligned
$$
\end{proof}
\section{The F-S inequality with arbitrary weights}
This section will be devoted to give the proof of Theorem 1.2. In order to demonstrate this theorem clearly, we consider more general setting, the multilinear strong maximal operator with complexity $c$. Theorem 1.2 follows immediately once the following estimate is proved:
$$\nu_{\vec{\omega}}\left({x\in\mathbb{R}^{n}:\mathcal{M}_{\mathcal{R}}^{c} (\vec{f})(x)> t^{m}}\right)^{1/p}\leq \prod^{m}_{j=1}\frac{C}{t}\|f_{j}\|_{L^{p_{j}}(\mathbb{R}^{n},W_{j})},$$
where $c=1,2...,n$ and $W_{j}$=$\mathcal{M} _{\mathcal{R}}^{c} \mathcal{M}_{\mathcal{R}}^{c-1} ...\mathcal{M}_{\mathcal{R}}^{1} \omega_{j}$, $j=1,2,...,m$.
The same selection procedure as in \cite{T} will be used in our proof. We only consider the bilinear case, the multilinear case can be obtained in the similar way easily. Moreover, we also need the following lemma.
\begin{lemma}[\cite{LOPTT}]
Let $\frac{1}{p}=\frac{1}{p_{1}}+\frac{1}{p_{2}}+...+\frac{1}{p_{m}}$ and $\nu_{\vec{\omega}}=\prod^{m}_{j=1}\omega_{j}^{\frac{p}{p_{j}}}$, if $1\leq p_{j}<\infty$, then for arbitrary weights $\omega_1,..,\omega_{j}$, it holds that
$$\|\mathcal{M}(\vec{f})\|_{L^{p,\infty}(\nu_{\vec{\omega}})}\leq c\prod^{m}_{j=1}\|f_{j}\|_{L^{p_{j}}(\mathcal{M}\omega_{j})}.$$\end{lemma}
Now,we give the proof of Theorem \ref{tm2}.
\begin{proof}
Notice that Theorem 1.2 holds for $c=1$. In fact, when $c=1$, $\mathcal{R}_{c}$ is the set of cubes, then Theorem 1.2 follows by Lemma 4.1. We assume that this theorem holds for $c=m-1$ and then we shall prove it for $c=m$. With a standard argument, we may assume that the basis $\mathcal{R}_{m}$ is the set of all dyadic rectangles(Cartesian products of dyadic intervals). We further assume that, when $R\in \mathcal{R}_{m}$, the sidelengths $|P_{i}(R)|$ decrease and
$$|P_{1}(R)|=|P_{2}(R)|=...=|P_{\hat{m}}(R)|>|P_{\hat{m}+1}(R)|.$$
For any compact set $K\subset\{x\in \mathbb{R}^{n}: \mathcal{M}_{\mathcal{R}}^{m-1} (\vec f)(x)>t^{2}\}$, there exist $\{R_{i}\}^{M}_{i=1}\subset \mathcal{R}_{m}$ such that $K\subset\bigcup^{M}_{i=1} R_{i}$ and
\begin{equation}\label{D}
\prod^{2}_{j=1}\frac{1}{|R_{i}|}\int_{R_{i}}|f_{j}|dy> t^{2}, \quad j=1,2,...,M.
\end{equation}
First, relabel if necessary so that the $R_{i}'s$ are ordered in a way such that their long sidelengths $|P_{1}(R_{i})|$ decrease. We now give a selection procedure to find subcollection $\{\tilde{R}_{i}\}^{N}_{i=1}\subset\{R_{i}\}^{M}_{i=1}$. \par
Take $\tilde{R}_{1}\doteq R_{1}$ and suppose that we have now chosen the rectangles $\tilde{R}_{1}, \tilde{R}_{2},...,\tilde{R}_{i-1}$. We select $\tilde{R}_{i}$ to be the first rectangle $R_{k}$ occurring after $\tilde{R}_{i-1}$ so that
$$\left|\bigcup^{i-1}_{j=1}\tilde{R}_{j}\cap{R}_{k}\right|<\frac{1}{2}|{R}_{k}|,\quad i=2,3,...,N.$$
Thus, $\tilde{R}_{i}$ enjoys the property that
\begin{equation}\label{EEE}
\left|\bigcup^{i-1}_{j=1}\tilde{R}_{j}\cap\tilde{R}_{i}\right|<\frac{1}{2}|\tilde{R}_{i}|.
\end{equation}
Set $\Omega\doteq\bigcup^{N}_{i=1}\tilde{R}_{i}$. We claim that
\begin{equation}\label{E}
\bigcup^{M}_{i=1}R_{i}\subset\left\{x\in\mathbb{R}^{n}:\mathcal{M}_{\mathcal{R}}^{m-1} (\mathbf{1}_{\Omega},\mathbf{1}_{\Omega})(x)> \frac{1}{2^{2}}\right\}.
\end{equation}
Indeed, choose any point $x$ inside a rectangle $R_{j}$ that is not one of the selected rectangles $\tilde{R}_{i}$.
Then, there exists a unique $J\leq N$ such that
$$
\left|\bigcup^{J}_{j=1}\tilde{R}_{j}\cap R_{j}\right|\geq \frac{1}{2}|R_{j}|.
$$
Since, $|P_{l}(\tilde{R}_{i})|\geq|P_{l}(R_{j})|$ for $l=1, 2,...,\hat{m}$ and $i=1,2,...,J$, if $\tilde{R}_{i}\cap R_{j}\neq \emptyset$, we have
$$P_{l}(\tilde{R}_{i})\cap P_{l}(R_{j})=P_{l}(R_{j}).$$Therefore, we obtain
\begin{align*}
\bigcup^{J}_{j=1}\tilde{R}_{j}\cap R_{j}
&=\bigcup^{J}_{j=1}\left(\prod^{\hat{m}}_{i=1}P_{l}(R_{j})\right)\times \left(\prod^{n}_{l=\hat{m}+1}P_{l}(\tilde{R}_{i})\cap P_{l}(R_{j})\right)\\
&=\left(\prod^{\hat{m}}_{i=1}P_{l}(R_{j})\right)\times \bigcup^{J}_{i=1}\left(\prod^{n}_{l=\hat{m}+1}P_{l}(\tilde{R}_{i})\cap P_{l}(R_{j})\right)
\end{align*}
Hence,
$$\left|\bigcup^{J}_{i=1}\left(\prod^{n}_{l=\hat{m}+1}P_{l}(\tilde{R}_{i})\cap P_{l}(R_{j})\right)\right|\geq \frac{1}{2}\left|\prod^{n}_{l=\hat{m}+1}P_{l}(R_{j})\right|.$$
Thanks to the fact that $|P_{\hat{m}+1}(R_{j})|<|P_{\hat{m}}(R_{j})|$, this implies that
$$\left|\bigcup^{K}_{i=1}\tilde{R}_{i}\cap R\right|\geq \frac{1}{2}|R|,$$
where $R$ is a unique dyadic rectangle containing $x$ and satisfies
$$|P_{1}(R)|=|P_{2}(R)|=...=|P_{\hat{m}}(R)|=|P_{\hat{m}+1}(R_{j})|.$$
This proves (\ref{E}), by the reason that such $R$ should belong to $\mathcal{R}_{m-1}$.
From this, we get
\begin{align}
\nu_{\vec{\omega}}\left(\bigcup^{M}_{i=1}R_{i}\right)^{1/p}
&\leq \nu_{\vec{\omega}}(\{x\in\mathbb{R}^{n}:\mathcal{M}_{\mathcal{R}}^{m-1}  (\mathbf{1}_{\Omega},\mathbf{1}_{\Omega})(x)> 1/2^{2}\})^{1/p}\\
&\nonumber \leq C\prod^{2}_{j=1}\|\mathbf{1}_{\Omega}\|_{L^{p_{j}}(\mathbb{R}^{n},U_{j})},
\end{align}
where $U_{j}=\mathcal{M} _{\mathcal{R}}^{m-1} \mathcal{M}_{\mathcal{R}}^{m-2} ...\mathcal{M}_{\mathcal{R}}^{1} \omega_{j}$.
Set $E(\tilde{R}_{1})=\tilde{R}_{1}$. For $i=2,3,...,N$, set
$$E(\tilde{R}_{i})=\tilde{R}_{i}\setminus\bigcup^{i-1}_{j=1}\tilde{R}_{j}.$$
Then, the sets $E(\tilde{R}_{i})$ are pairwise disjoint and by (\ref{EEE}), it holds that
\begin{equation}\label{C}
|E(\tilde{R}_{i})|\geq \frac{1}{2}|\tilde{R_i}|, \quad i=1,2,...,N.
\end{equation}
Thus,
$$\prod^{2}_{j=1}\|\mathbf{1}_{\Omega}\|_{L^{p_{j}}(\mathbb{R}^{n},U_{j})}=\prod^{2}_{j=1}U_{j}(\Omega)^{\frac{1}{p_{j}}}\leq \prod^{2}_{j=1}\left(\sum^{N}_{i=1}U_{j}(\tilde{R}_{i})\right)^{\frac{1}{p_{j}}}.$$
Hence by (\ref{D}), one may obtain that
\begin{align*}
&\prod^{2}_{j=1}\left(\sum^{N}_{i=1}U_{j}(\tilde{R}_{i})\right)^{\frac{1}{p_{j}}}\times 1\\
&\leq\prod^{2}_{j=1}\left(\sum^{N}_{i=1}U_{j}(\tilde{R}_{i})\right)^{\frac{1}{p_{j}}}\times \prod^{2}_{j=1}\frac{1}{t|\tilde{R}_{i}|}\int_{\tilde{R}_{i}}|f_{j}(y)|dy\\
&\leq\frac{1}{t^{2}}\prod^{2}_{j=1}\left(\sum^{N}_{i=1}U_{j}(\tilde{R}_{i})\left(\frac{1}{|\tilde{R}_{i}|}\int_{\tilde{R}_{i}}|f_{j}(y)|dy\right)^{p_{j}}\right)^{\frac{1}{p_{j}}}
\end{align*}
Note that
\begin{align*}
&\sum^{N}_{i=1}U_{j}(\tilde{R}_{i})\left(\frac{1}{|\tilde{R}_{i}|}\int_{\tilde{R}_{i}}|f_{j}(y)|dy\right)^{p_{j}}\\
&=\sum^{N}_{i=1}\left(\frac{1}{|\tilde{R}_{i}|}\int_{\tilde{R}_{i}}|f_{j}|dy\left(\frac{1}{|\tilde{R}_{i}|}\int_{\tilde{R}_{i}}U_{j}dy\right)^{\frac{1}{p_{j}}}\right)^{p_{j}}|\tilde{R}_{i}|\\
&\leq 2\sum^{N}_{i=1}\left(\frac{1}{|\tilde{R}_{i}|}\int_{\tilde{R}_{i}}|f_{j}|W_{j}^{\frac{1}{p_{j}}}dy\right)^{p_{j}}|E(\tilde{R}_{i})|\\
&\leq 2\int_{\mathbb{R}^{n}}(\mathcal{M}_{m}[f_{j}W_{j}^{1/p_{j}}])^{p_{j}}dy\\
&\leq C\int_{\mathbb{R}^{n}}|f_{j}|^{p_{j}}W_{j}dx,
\end{align*}
Combining them together with (4.4) ,we have
$$\nu_{\vec{\omega}}\left(\bigcup^{M}_{i=1}R_{i}\right)^{1/p}\leq \frac{C}{t^{2}}\prod^{2}_{j=1}\left(\int_{\mathbb{R}^{n}}|f_{j}|^{p_{j}}W_{j}dx\right).$$
where we have used (\ref{C}), and the $L^{p}$-boundedness of $\mathcal{M}_{m}$. Altogether, we obtain the desired result.
\end{proof}

\end{document}